\documentclass[]{amsart}
\usepackage{amscd,amsthm,amssymb,amsfonts,amsmath,euscript}

\theoremstyle{plain}
\newtheorem{thm}{Theorem}[section]
\newtheorem{lemma}[thm]{Lemma}
\newtheorem{prop}[thm]{Proposition}
\newtheorem{cor}[thm]{Corollary}

\theoremstyle{definition}
\newtheorem{defn}[thm]{Definition}

\theoremstyle{remark}

\newtheorem*{thank}{{\bf Acknowledgments}}
\newcommand{\nc}{\newcommand}

\def\makeop#1{\expandafter\def\csname#1\endcsname
  {\mathop{\rm #1}\nolimits}\ignorespaces}
\makeop{Hom}   \makeop{End}   \makeop{Aut}   \makeop{Isom}  \makeop{Pic} 
\makeop{Gal}   \makeop{ord}   \makeop{Char}  \makeop{Div}   \makeop{Lie} 
\makeop{PGL}   \makeop{Corr}  \makeop{PSL}   \makeop{sgn}   \makeop{Spf}
\makeop{Spec}  \makeop{Tr}    \makeop{Nr}    \makeop{Fr}    \makeop{disc}
\makeop{Proj}  \makeop{supp}  \makeop{ker}   \makeop{im}    \makeop{dom}
\makeop{coker} \makeop{Stab}  \makeop{SO}    \makeop{SL}    \makeop{SL}
\makeop{Cl}    \makeop{cond}  \makeop{Br}    \makeop{inv}   \makeop{rank}
\makeop{id}    \makeop{Fil}   \makeop{Frac}  \makeop{GL}    \makeop{SU}
\makeop{Trd}   \makeop{Sp}    \makeop{Tr}    \makeop{Trd}   \makeop{diag}
\makeop{Res}   \makeop{ind}   \makeop{depth} \makeop{Tr}    \makeop{st}
\makeop{Ad}    \makeop{Int}   \makeop{tr}    \makeop{Sym}   \makeop{can}
\makeop{length}\makeop{SO}    \makeop{torsion} \makeop{GSp} \makeop{Ker}
\makeop{Adm}
\def\makebb#1{\expandafter\def
  \csname bb#1\endcsname{{\mathbb{#1}}}\ignorespaces}
\def\makebf#1{\expandafter\def\csname bf#1\endcsname{{\bf
      #1}}\ignorespaces} 
\def\makegr#1{\expandafter\def
  \csname gr#1\endcsname{{\mathfrak{#1}}}\ignorespaces}
\def\makescr#1{\expandafter\def
  \csname scr#1\endcsname{{\EuScript{#1}}}\ignorespaces}
\def\makecal#1{\expandafter\def\csname cal#1\endcsname{{\mathcal
      #1}}\ignorespaces} 

\def\doLetters#1{#1A #1B #1C #1D #1E #1F #1G #1H #1I #1J #1K #1L #1M
                 #1N #1O #1P #1Q #1R #1S #1T #1U #1V #1W #1X #1Y #1Z}
\def\doletters#1{#1a #1b #1c #1d #1e #1f #1g #1h #1i #1j #1k #1l #1m
                 #1n #1o #1p #1q #1r #1s #1t #1u #1v #1w #1x #1y #1z}
\doLetters\makebb   \doLetters\makecal  \doLetters\makebf
\doLetters\makescr 
\doletters\makebf   \doLetters\makegr   \doletters\makegr
     \def\qed{\qedmark\medbreak}%
\def\qedmark{{\enspace\vrule height 6pt width 5pt depth 1.5pt}}%

\normalsize

\makeop{Bl}

\def\Fpbar{\overline{\bbF}_p}
\def\Fp{{\bbF}_p}

\def\Qp{{\bbQ}_p}

\def\Zp{{\bbZ}_p}

\def\N{{\bbN}}

\newcommand{\Z}{\mathbb Z}
\newcommand{\Q}{\mathbb Q}
\newcommand{\R}{\mathbb R}

\newcommand{\A}{\mathbb A}    
\newcommand{\F}{\mathbb F}

\nc{\embed}{\hookrightarrow}

\newcommand{\ch}{characteristic }
\newcommand{\ac}{algebraically closed }
\newcommand{\dieu}{Dieudonn\'{e} }

\nc{\ol}{\overline}
\nc{\wt}{\widetilde}
\nc{\opp}{\mathrm{opp}}

\makeop{Ram}
\makeop{Rep}

\begin{document}
\renewcommand{\thefootnote}{\fnsymbol{footnote}}
\setcounter{footnote}{-1}
\numberwithin{equation}{section}


\title[Superspecial abelian varieties]
{Superspecial abelian varieties over finite prime fields}
\author{Chia-Fu Yu}
\address{
Institute of Mathematics, Academia Sinica and NCTS (Taipei Office)\\
6th Floor, Astronomy Mathematics Building \\
No. 1, Roosevelt Rd. Sec. 4 \\ 
Taipei, Taiwan, 10617} 
\email{chiafu@math.sinica.edu.tw}

\date{\today}

\begin{abstract}
In this paper we determine the number of isomorphism classes of
superspecial abelian varieties $A$ over the prime field $\Fp$ 
such that the relative Frobenius morphism $\pi_A$ satisfying 
$\pi_A^2=-p$. 
\end{abstract} 

\maketitle


\section{Introduction}
\label{sec:01}

Let $p$ be a rational prime number and $g\ge 1$ be a positive
integer. An abelian variety over a field $k$ of \ch $p$ is said to be 
{\it superspecial} if it is isomorphic to a product of supersingular
elliptic curves over an algebraic closure $\bar k$ of $k$. Let $E_0$
be a supersingular elliptic curve over $\Fp$ such that
$\pi_{E_0}^2=-p$, where $\pi_{E_0}$ is the relative Frobenius
endomorphism of $E_0$. Such an elliptic
curve exists (see Deuring \cite{deuring} or use the Honda-Tate theory
\cite[Theorem 1, p.~96]{tate:ht})
and the condition $\pi_{E_0}^2=-p$ is automatic if $p>3$;
the latter follows from the Hasse-Weil bound for eigenvalues of the
Frobenius. Let $\scrS$ be the set of isomorphism
classes of $g$-dimensional superspecial abelian
varieties $A$ over $\Fp$ such that there is an isogeny from
$E_0^g$ to $A$ over $\Fp$. Using a theorem of Tate \cite[Theorem 1
(c), p.~139]{tate:eav}, the
condition for $A$ isogenous to $E_0^g$ over $\Fp$ is equivalent to that 
the relative Frobenius morphism
$\pi_{A}$ of $A$ over $\Fp$ satisfies $\pi_{A}^2=-p$ 
(also see Lemma~\ref{22}). 
In this paper we calculate the number of these
superspecial abelian varieties. 

\begin{thm}\label{11}
  Notation as above, we have 
\begin{equation}
    \label{eq:11}
 |\scrS|=
\begin{cases}
  h(\sqrt{-p}), & \text{if $p=2$ or $p\equiv 1\ (\,{\rm mod}\, 4)$},
  \\ 
  (g+1) h(\sqrt{-p}), & \text{if $p\equiv 7 \ (\,{\rm mod}\, 8)$ or
  $p=3$}, \\ 
  (g+3) h(\sqrt{-p}), & \text{if $p\equiv 3 \ (\,{\rm mod}\, 8)$ and
  $p\neq 3$,}   
\end{cases}    
\end{equation}
where $h(\sqrt{-p})$ denotes the class number of the imaginary quadratic
field $\Q(\sqrt{-p})$.
\end{thm}

It remains to discuss some background on the topic. First of all, it
is well-known (due to Deuring \cite{deuring}) that every supersingular
elliptic curve over an \ac field of \ch $p$ has a model defined over
$\F_{p^2}$. Also, if we let  
$B_{p,\infty}$
denote the quaternion algebra over $\Q$ ramified exactly at $p$ and
$\infty$, then the set of isomorphism
classes of supersingular elliptic curves over $\Fpbar$ is in one-to-one
correspondence with the set of ideal classes of a maximal order of
the quaternion algebra $B_{p,\infty}$. 
The same picture can be generalized to higher
dimensions: (a) there is only one isomorphism classes of superspecial
abelian varieties of dimension $>1$ over $\Fpbar$ (this is due to
Deligne, Ogus and Shioda), and 
(b) the set $\Lambda_g$ of isomorphism classes of
superspecial principally polarized abelian varieties over $\Fpbar$ has
the similar description as a double coset space for 
an algebraic group $G$ associated to certain quaternion hermitian form 
(see Ibukiyama-Katsura-Oort
\cite[Theorem~2.10]{ibukiyama-katsura-oort}). 
The class number $|\Lambda_g|$ is
calculated by Deuring \cite{deuring,deuring:jdm50} and Eichler \cite{eichler} 
for $g=1$ and by Hashimoto and Ibukiyama
\cite{hashimoto-ibukiyama:classnumber} for $g=2$. 
It is believed after \cite{hashimoto-ibukiyama:classnumber} that
calculating the class number $|\Lambda_g|$ for higher genus $g$ is 
an extremely difficult task. However, exploring some structures 
and relationships among them arising from $\Lambda_g$ is still 
interesting. 

The well-known Deuring-Eichler mass formula suggests 
that instead of calculating
the class number $|\Lambda_g|$ itself, the weighted version, ${\rm
  Mass}(\Lambda_g)$, in which one associates each object $(A,\lambda)$
to the weight $\# \Aut(A,\lambda)^{-1}$ and sums over the objects in 
$\Lambda_g$, 
should be more accessible (see \cite[Introduction]{yu:smf} for a
discussion). 
This is done for $\Lambda_g$, namely the case of Siegel modular 
varieties, 
by Ekedahl \cite[p.~159]{ekedahl:ss} and some others (see
Hashimoto-Ibukiyama 
\cite[Proposition 9, p.~568]{hashimoto-ibukiyama:classnumber}, 
and 
Katsura-Oort \cite[Section 2, Theorems~5.1 and 5.3]{katsura-oort:surface}). 
The analogous formula for the mass associated to the set
of superspecial polarized abelian varieties with real multiplication 
is showed in \cite[Theorem~3.7 and Subsection~4.6]{yu:mass_hb}. 
This result is applied for determining
the number of supersingular components (those are entirely contained
in the supersingular locus) of the reduction of Hilbert
modular varieties with Iwahori level structure; see
\cite[Section~4]{yu:gamma_hb}.  
The geometric mass formula is generalized to good 
reduction of quaternion Shimura varieties; see
\cite[Theorem~1.2]{yu:gmf} for the precise formula. 
The proof of this mass formula uses an arithmetic
mass formula obtained by Shimura \cite[Introduction, p.~68]{shimura:1999} 
and the connection between geometric
and arithmetic masses (see \cite[Theorem~2.2]{yu:smf} for the precise
statement). 

In the function field analogue where
superspecial abelian varieties are replaced by supersingular Drinfeld
modules, the mass formula is obtained by J. Yu and the author
\cite[Theorem~2.1, p.~906]{yu-yu:ssd} (for any rank and function fields),
based on some earlier works of Gekeler 
\cite{gekeler:finite, gekeler:mass}. 

Another viewpoint to depart
concerning counting superspecial points (still over $\Fpbar$) is
that the superspecial locus itself is an $\ell$-adic Hecke orbit, where
$\ell$ is a prime $\neq p$, in the fine moduli space 
(see Chai \cite{chai:ho} for the definition and results about
$\ell$-adic Hecke orbits). 
It turns out that the similar formalism allows one to calculate
the cardinality of each {\it supersingular} Hecke orbit, 
provided one knows the underlying $p$-divisible group structure 
explicitly. The case of genus
$g=2$ is analyzed in J.-D. Yu and the author 
\cite[Theorem~1.1]{yu-yu:mass_surface}
using the Moret-Bailly family \cite{moret-bailly:p1}. 

Theorem~\ref{11} is not a weighted version. It sits between the
trivial case (a) and the extremely difficult case (b) above. 
It is worth noting that the set $\scrS$ is not always a single
$\ell$-adic Hecke orbit over $\Fp$ (here two abelian varieties over 
$\Fp$ lie in the same $\ell$-adic Hecke
orbit over $\Fp$ if there is an $\ell$-quasi-isogeny from one to 
another over $\Fp$), rather it consists of $g+1$ Hecke orbits in 
some cases (see Proposition~\ref{53}). 
A crucial ingredient which makes our
situation simpler is that the strong approximation holds for the
algebraic group $R_{E/\Q} \SL_{n,E}$ (see Lemma~\ref{51}), 
where $E$ is a number field and $n>1$.  
Note that we impose a condition on $\scrS$ which also makes
the computation simpler. As any superspecial abelian variety $A$ is
isogenous (even isomorphic) to $E_0^g$ over $\Fpbar$, this condition is
un-seen the geometric setting. 
Suppose that one does not impose this isogenous condition, and let
$\scrS'$ be the set of isomorphism classes of $g$-dimensional 
superspecial abelian varieties over $\Fp$. 
In the case $g=1$ Theorem~\ref{11} gives the size of $\scrS'$ for
$p>3$, as there is only one isogeny class in $\scrS'$. 
It is not hard to treat the cases $p=2$ and $3$ separately, and we get
\begin{equation}
  \label{eq:12}
  |\scrS'|=
  \begin{cases}
    2 h(\sqrt{-1})+h(\sqrt{-2}), & \text{if $p=2$}, \\
    4 h(\sqrt{-3}),  & \text{if $p=3$}.
  \end{cases}
\end{equation}
For the genus $g>1$, the
set $\scrS'$ is classified by the cohomology 
\begin{equation}
  \label{eq:13}
 H^1(\Gal(\Fpbar/\Fp), \Aut_{\Fpbar}(E_0^g)),  
\end{equation}
due to the fact (a). This is a single example
of $k$-forms of quasi-projective algebraic varieties over a perfect
field $k$, and this problem may not be very interesting, 
unless one finds some structure on (\ref{eq:13}) which allows us to
compute it more effectively.    

The proof of Theorem~\ref{11}
uses the classification of modules over certain non-maximal order $R$
of a number field $E$. The classification is of
interest on its own right; on the other hand, that is also useful to
determine isomorphism classes in other isogeny classes over $\Fp$ (see
Theorem~\ref{41}).   

Another closely related problem studied is about the field of
definition of isomorphism classes over $\Fpbar$. Deuring showed that
the number $h'$ of isomorphism classes of supersingular elliptic
curves over $\Fpbar$ which have a model over $\Fp$ is as follows 
(see Deuring \cite{deuring:jdm50}, also c.f. 
Ibukiyama-Katsura~\cite[Remark 3, p.~42]{ibukiyama-katsura}) 
For simplicity, one has for $p>3$, 
\begin{equation}
  \label{eq:14}
  h'=
  \begin{cases}
    \frac{1}{2} h(\sqrt{-p}), & \text{if $p\equiv 1 \ (\,{\rm mod}\,
    4)$,} \\ 
    h(\sqrt{-p}), & \text{if $p\equiv 7 \ (\,{\rm mod}\, 8)$,} \\ 
    2h(\sqrt{-p}), & \text{if $p\equiv 3 \ (\,{\rm mod}\, 8)$.} \\
  \end{cases}
\end{equation}
The relationship between (\ref{eq:11}) and (\ref{eq:14}) for $g=1$ is
that each isomorphism class over $\Fpbar$ which has a model over
$\Fp$ has exactly two isomorphism classes over $\Fp$. Deuring
\cite{deuring:jdm50} also showed the following equality
\begin{equation}
  \label{eq:15}
  h'=2t-h,
\end{equation}
where $t$ is the type number of the quaternion algebra $B_{p,\infty}$
and $h$ is the class number of $B_{p,\infty}$. Ibukiyama and Katsura
further generalized the result of Deuring to higher genus $g$. 
They showed that the same
relation (\ref{eq:15}) holds for the number 
$H'$ of the objects in $\Lambda_g$ which
have a model over $\Fp$, the type number $T$ of the quaternion
unitary algebraic group $G$ in question, and the class number $H$ of
the group $G$; see \cite[Theorems 1 and 2]{ibukiyama-katsura}. 
However, as mentioned before, this is a relationship we 
could have, and no explicit formula for each term is given 
when $g>2$.   

\begin{thank}
  The author thanks Frans Oort for helpful comments on an earlier
  manuscript. 
  The manuscript is prepared during the
  author's stay at l'Institut des Hautes \'Etudes Scientifiques. 
  He acknowledges the institution for kind hospitality and excellent working
  conditions. The research was partially supported by grants NSC
  97-2115-M-001-015-MY3 and AS-99-CDA-M01.
\end{thank}

\section{Preliminaries}
\label{sec:02}

Let the set $\scrS$ and the elliptic curve $E_0$ be as in \S
~\ref{sec:01}. Let $\calG$ be the
Galois group $\Gal(\Fpbar/\Fp)$. We define a set $\Phi_v$ for each
prime $v$ of $\Q$. If $\ell$ is a prime $\neq p$, let
$\Phi_\ell$ be the set 
of isomorphism classes of Tate modules $T_\ell(A)$ as 
$\Z_\ell[\calG]$-modules for
all $A\in \scrS$. Let $\Phi_p$ be the set of isomorphism classes of
\dieu modules $M(A)$ for all $A\in \scrS$. In this paper we use 
covariant \dieu modules. We refer the reader to Demazure 
\cite{demazure} and
Manin \cite{manin:thesis} for a basic account of \dieu theory. 

Let $M$ be a \dieu module over a perfect field $k$ of \ch $p$. We
recall that the {\it $a$-number} of $M$, denoted by $a(M)$, is the
dimension 
of the $k$-vector space $M/(F,V)M$. The $a$-number of an abelian
variety over $k$, denoted by $a(A)$, is defined to be the $a$-number
$a(M(A))$ of the \dieu module $M(A)$ associated to $A$. 

\begin{thm}\label{21}
  Let $A$ be an abelian variety of dimension $g$ over an \ac field 
  of \ch $p$. If $a(A)=g$, then $A$ is superspecial. 
\end{thm}

\begin{proof}
This is Theorem 2 of Oort \cite{oort:product}.   
\end{proof}

\begin{lemma}\label{22}
  Let $A$ be a $g$-dimensional abelian variety over $\Fp$. Then $A$ is 
  isogenous to $E_0^g$ over $\Fp$ if and only if the relative
  Frobenius endomorphism $\pi_A$ of $A$ over $\Fp$ 
  satisfies $\pi_A^2+p=0$. Furthermore, in this case the abelian
  variety $A$ is superspecial.
\end{lemma}
\begin{proof}
  Using a theorem of Tate \cite[Theorem 1 (c), p.~139]{tate:eav}, 
  $A$ is isogenous to $E_0^g$ over $\Fp$
  if and only if the characteristic polynomial of the endomorphism
  $\pi_A$ is equal to that of $\pi_{E_0^g}$, which is $(X^2+p)^g$. On
  the other hand, since the Tate space $V_\ell(A):=T_\ell(A)\otimes
  \Q_\ell$ is semi-simple as $\Q_\ell[\calG]$-modules, 
  that the characteristic polynomial of $\pi_A$ 
  equal to $(X^2+p)^g$ implies that the minimal polynomial of $\pi_A$ is
  $X^2+p$. This shows the first statement. 

For the second statement, we use Theorem~\ref{21}. Since $F^2=-p$ on
  $M(A)$, we get $VM(A)=FM(A)$ and hence $a(A)=g$. This proves the
  lemma \qed 
\end{proof}

Put $R:=\Z[X]/(X^2+p)=\Z[\pi]$ and $E:=R\otimes \Q=\Q(\sqrt{-p})$. Let
$O_E$ 
be the ring of integers of $E$. 
For each finite place $v$ of $\Q$, write $R_v$, $E_v$ and $O_{E_v}$
for $R\otimes_{\Z} \Z_v$, $E\otimes_{\Q} \Q_v$, and $O_E \otimes_{\Z}
\Z_v$, respectively.

Let $A$ be an object in $\scrS$. 
Let $\sigma_p\in \calG$ be the arithmetic
Frobenius automorphism $x\mapsto x^p$.
We have $\sigma_p x=\pi_A x$ for all $x\in T_\ell(A)$, where $\ell$ is 
any prime $\neq p$.  
Since $\pi_A^2+p=0$, the action of $\calG$ on $T_\ell(A)$ factors
through the epi-morphism $\Z_\ell[\calG]\to \Z_\ell[X]/(X^2+p)$. 
Therefore, we may classify Tate modules $T_\ell(A)$ as
$R_\ell$-modules. 
On the other hand the Tate space $V_\ell(A)$ is a free
$E_\ell$-module of rank $g$; this follows from the fact that 
$\tr(a; V_\ell(A))= g \tr (a; E_\ell)$ for all $a\in R$. 
From this, if $R_\ell$
is the maximal order of $E_\ell$, then $T_\ell(A)$ is a free
$R_\ell$-module of rank $g$. 
In this case the set $\Phi_\ell$ consists of single element. 

In the case $v=p$, the \dieu module $M(A)$ is a
free $\Z_p$-module of rank $2g$, together with a $\Zp$-linear operator $F$
satisfying $F^2+p$, and hence $M(A)$ is simply a $\Z_p$-free finite 
$R_p$-module. Since
$R_p\otimes \Q_p=\Q_p(\sqrt{-p})=E_p$ is a field and $R_p$ is always 
the maximal
order of $E_p$, the \dieu module $M(A)$ is $R_p$-free of rank
$g$. This shows that the set $\Phi_p$ also consists of single element. 

As an elementary result, 
the ring $R_v$ is not the maximal order if and only if 
$v=2$ and $p\equiv 3 \ (\,{\rm mod}\, 4)$. We conclude

\begin{lemma}\label{23}
  The set $\Phi_v$ consists of single element except when $v=2$ and
  $p\equiv 3 \ (\,{\rm mod}\, 4)$. 
\end{lemma}

\section{Classification of $\Phi_2$}
\label{sec:03}
In this section we assume that $p\equiv 3 \ (\,{\rm mod}\, 4)$. 
To simplify the 
notation, we write $R:=\Z_2[X]/(X^2+p)=\Z_2[\pi]$, 
$E:=R\otimes_{\Z_2} \Q_2$ and $O_E$
for the ring of integers of $E$. We have 
\[ O_E=\Z_2[\alpha]=\Z_2[X]/(X^2+X+(p+1)/4). \]
Put $\omega:=\pi-1$, and one has 
\[ R=\Z_\ell[\omega]=\Z_2[X]/(X^2+2X+(p+1)\,)\quad \text{and}
\quad 2\alpha=\omega.  \] 
We shall classify $R$-modules $M$ which is finite and free as
$\Z_2$-modules. We divide the classification into two cases:\\

{\bf Case (a):} $p\equiv 3 \ (\,{\rm mod}\, 8)$. In this case, $E$ is
a unramified 
quadratic extension of $\Q_2$. We have (at least) 
two indecomposable $\Z_2$-free finite $R$-modules: 
$R$ and $O_E$ as $R$-modules. The $R$-module structure of $O_E$ is
given as follows: write $O_E=<1,\alpha>_{\Z_2}$, then 
\begin{equation}
  \label{eq:31}
  \omega 1= 2\alpha\quad\text{and} \quad 
\omega \alpha= -2\alpha - (p+1)/2. 
\end{equation}
If $M=R^r\oplus O_E^s$, then the non-negative integers $r$ and $s$ are
uniquely determined by $M$. 
Indeed, we have $r+s=\dim_{E} M\otimes_{\Z_2} \Q_2$, 
and $M/(2,\omega)M=(\F_2)^r\otimes (\F_2\oplus \F_2)^s. $  \\

{\bf Case (b):} $p\equiv 7 \ (\,{\rm mod}\, 8)$. In this case,
$E=\Q_2\times 
\Q_2$. Write 
\[ X^2+X+(p+1)/4=(X-\alpha_1)(X-\alpha_2), \]
where
$\alpha_1, \alpha_2\in \Z_2$. By switching the order, we may assume
that $\alpha_1$ is a unit and $\alpha_2\in 2\Z_2$. We have the isomorphisms
\[ O_E=\Z_2[\alpha] \simeq O_E/(\alpha-\alpha_1)\times
O_E/(\alpha-\alpha_2)\simeq \Z_2\times \Z_2. \] 
Therefore, 
\[ X^2+2X+(p+1)=(X-2\alpha_1)(X-2\alpha_2). \]
We have (at least) three indecomposable $\Z_2$-free finite
$R$-modules: 
\[ R, \quad R/(\omega-2\alpha_1),\quad\text{and}\quad R/(\omega-2\alpha_2).\] 
Among them, we have 
\[ O_E\simeq R/(\omega-2\alpha_1)\oplus R/(\omega-2\alpha_2) \]
as $R$-modules. 
If $M=R^r\oplus [R/(\omega-2\alpha_1)]^s\oplus
[R/(\omega-2\alpha_2)]^t$, then the non-negative integers $r$, $s$ and
$t$ are uniquely determined by $M$. Indeed, we have 
\[ \rank_{\Z_2} M=2r+s+t, \quad 
M/(2,\omega)M=\F_2^r\oplus \F_2^s\oplus \F_2^t, \]
and  
\[ M/(\omega-2\alpha_1)M=[R/(\omega-2\alpha_1)]^{r+s}\oplus (\F_2)^t. \]\


Conversely, we show that the indecomposable
finite $R$-modules described in
Cases (a) and (b) exhaust all possibilities.
  
\begin{thm}\label{31} 
Let $M$ be a $\Z_2$-free finite $R$-module. Then \
\begin{enumerate}
  \item {\rm Case (a)}. The $R$-module $M$ is isomorphic to
    $R^r\otimes O_E^s$ 
    for some non-negative integers $r$ and $s$. Moreover, the integers 
    $r$ and $s$ are uniquely determined by $M$. 
  \item {\rm Case (b)}. The $R$-module $M$ is isomorphic to 
\[ R^r\oplus \left [ R/(\omega-2\alpha_1)\right ]^s\oplus
\left [ R/(\omega-2\alpha_2)\right ]^t\]
    for some non-negative integers $r$, $s$ and $t$. Moreover, the
    integers 
    $r$, $s$ and $t$ are uniquely determined by $M$. 
\end{enumerate}
\end{thm}
\begin{proof} The unique determination of integers $r$, $s$ and $t$
    has been showed. We prove the first part of each statement. 

  (1) Let 
  \begin{equation}
    \label{eq:32}
    \ol M:=M/2M=(\F_2[\omega]/\omega^2)^r\oplus (\F_2)^{2s}
  \end{equation}
be the decomposition as $R/2R=\F_2[\omega]/\omega^2$-modules. We
first show that if $s=0$, then $M\simeq R^r$. Since $r=\dim
M\otimes_{R} \F_2=\dim_E M\otimes_{R} E$ and $R$ is a local Noetherian
domain, the module $M$ is free. 

Now suppose $s>0$. Choose an element $\bar a \neq 0 \in (\F_2)^{2s}$
and let 
$x\in M$ be an element such that $\bar x=\bar a$. As $\ol{\omega
  x}=0$, the element $\omega x/2\in M$. Put $M_1:=<x,\omega
x/2>_{\Z_2}$; it is an $R$-module and is isomorphic to $O_E$. Let
$\omega'$ be the conjugate of $\omega$, one has $\omega'=-2-\omega$
and $\omega \omega'=(1+p)$. Note that $(1+p)/4$ is a unit. Since $\bar
x\not\in \omega' \ol M$, one has $x\not\in \omega' M$. We show that
$\ol{\omega x/2}\neq 0$. Suppose not, then $\omega x=4y$ for some
$y\in M$. Applying $\omega'$, we get $x=\omega' y'$ for some $y'\in
M$, contradiction. Since $x$ and ${\omega x/2}$ are $\Z_2$-linearly 
independent, the $\F_2$-vector space $\ol M_2=<\bar x, \ol{\omega
  x/2}>$ has 
dimension $2$, and hence the quotient $\ol M/\ol M_1$ has dimension
deceased by $2$. On the other hand, the $\Z_2$-rank of $M/M_1$ also
decreases by $2$. This shows that $M/M_1$ is free as $\Z_2$-modules. 
If the integer $s$ in (\ref{eq:32}) for $M/M_1$ is positive, then we
can find an $R$-submodule $M_2=<x_2, \omega x_2/2>_{\Z_2}\simeq O_E$
not contained in the vector space $E M_1$ 
such that $M/(M_1+M_2)$ is free as $\Z_2$-modules. 
Continuing this process, we get $R$-submodules $M_1,\dots, M_{s'}$,
which are isomorphic to $O_E$, such that $M_1+\dots+M_{s'}=M_1\oplus
\dots\oplus M_{s'}$ and $M/(M_1+\dots + M_{s'})$ is a free
$R$-module. It follows that $M\simeq O_E^{s'}\oplus R^{r'}$. Since
$s'$ and $r'$ are uniquely determined by $M$ as before, the integers
$s'$ and $r'$ are actually equal to $s$ and $r$ in (\ref{eq:32}), 
respectively. This proves (1). 

(2) Let 
\[ M_1:=\{x\in M\, |\, (\omega-2\alpha_1)x=0\, \},\]
and 
\[ M_1:=\{x\in M\, |\, (\omega-2\alpha_2)x=0\, \}.\]
Using the relation
\[
2=(\omega-2\alpha_1)(2\alpha_2+1)^{-1}-
(\omega-2\alpha_2)(2\alpha_2+1)^{-1}, \] 
one shows that $2M\subset M_1+M_2$, and hence the quotient
$M/(M_1+M_2)$ is an $\F_2$-vector space, say of dimension $r$.   
Let $x_1, \dots, x_r$ be elements of $M$ such that the images $\bar
x_1, \dots, \bar x_r$ form an $\F_2$-basis for $M/(M_1+M_2)$. Put
$F_0:=<x_1, \dots, x_r>_R$, which is isomorphic to $R^r$, as $\bar
x_i's$ form a basis for $F_0/(M_1+M_2)=F_0/(2,\omega)F_0\simeq
\F_2^r$. Now $(\omega-2\alpha_2)F_0\subset M_1$, we choose elements
$y_1, \dots, y_s$ in $M_1$ so that the images $\bar y_1, \dots, \bar
y_s$ form an $R/(\omega-2\alpha_1)$-basis for
$M_1/(\omega-2\alpha_2)F_0$, and put $F_1=<y_1,\dots, y_s>_R$. We have 
\[ M_1=(\omega-2\alpha_2)F_0\oplus F_1,\quad \text{and} \quad F_0\cap
F_1=0. \]
 Similarly, we have a free $R/(\omega-2\alpha_2)$-submodule 
$F_2$ of $M_2$, of rank $t$, such that
\[ M_2=(\omega-2\alpha_1)F_0\oplus F_2,\quad \text{and} \quad F_0\cap
F_2=0. \]  
We have $(F_0+F_1)\cap F_2=F_0\cap F_2=0$ and $M=F_0+F_1+F_2$, and
hence $M=F_0\oplus F_1\oplus F_2$. This proves (2). \qed  
\end{proof}

We retain the notation as in \S~1 and 2.

\begin{cor}\label{32}
  Assume $p\equiv 3 \ (\,{\rm mod}\, 4)$, then the Tate module
  $T_2(A)$ of an 
  object $A$ in $\scrS$ is isomorphic to $R_2^r\oplus O_{E_2}^s$ for
  some non-negative integers $r$ and $s$ such that $r+s=g$. Moreover,
  the integers $r$ and $s$ are uniquely determined by $T_2(A)$. 
\end{cor}

\begin{proof}
  Since the Tate space $V_2(A)$ is a free $\Q_2[\alpha]$-module, the
  numbers $s$ and $t$ in Theorem~\ref{31} (2) above are the
  same. Therefore, the corollary follows. \qed 
\end{proof}

\begin{lemma}\label{33}
  Assume $p\equiv 3 \ (\,{\rm mod}\, 4)$. 
  For any non-negative  integers $r$ and $s$ with $r+s=g$, there
  exists an abelian variety $A_r$ in $\scrS$ such that the Tate
  module $T_2(A_r)$ of $A_r$ is isomorphic to $R_2^r\oplus O_{E_2}^s$. 
\end{lemma}
\begin{proof}
  Choose a supersingular elliptic curve $E_0$ over $\F_p$ such that
  the endomorphism ring $\End_{\Fp}(E_0)$ is equal to $O_E$, and a
  supersingular elliptic curve $E_1$ over $\F_p$ such that 
  the endomorphism ring $\End_{\Fp}(E_0)$ is equal to $R$ (see
  Waterhouse \cite[Theorem 4.2 (3), p.~539]{waterhouse:thesis}). 
  Put $A_r=E_1^r\times E_0^s$, then the superspecial
  abelian variety $A_r$ has the desired property.\qed 
\end{proof}

\section{Abelian varieties over $\Fp$}
\label{sec:04}
\def\qisog{\text{Q-isog}}
\def\isog{\text{Isog}}
For our purpose we need to describe abelian varieties up to
isomorphism over $\Fp$.
The Honda-Tate theory \cite{tate:eav, tate:ht}
has described isogeny classes of abelian varieties over finite
fields. Therefore, we may focus on isomorphism classes in one single
isogeny class over $\Fp$. We describe this in terms of modules so that
we can count them explicitly.
In this section the ground field considered is $\Fp$. 

Let $A_0$ be a fixed abelian variety, and denote by ${\rm Isog}(A_0)$
the set of isomorphism classes in the isogeny class of $A_0$.  
Recall that an {\it quasi-isogeny}
$\varphi:A\to A_0$ is an element $\varphi\in \Hom(A,A_0)\otimes \Q$
such that $n \varphi$ is an isogeny for some integer $n$. 
We identify two quasi-isogenies $\varphi_i:A_i\to A_0$, $i=1,2$ as the
same element if there is an (necessarily unique) isomorphism $\rho:
A_1\to A_2$ such that $\varphi_2\circ \rho=\varphi_1$. Therefore, 
it makes sense to talk about the {\it set} of quasi-isogenies 
to $A_0$, which we denote $\qisog(A_0)$. The set $\qisog(A_0)$ 
can be parametrized by pairs $(H,n)$ 
where $H$ is a subgroup scheme of the dual abelian variety $A_0^t$ and
$n$ is a positive integer. The quasi-isogeny represented by the pair
$(H,n)$ is ``$(1/n)\pi^t$'', 
where $\pi:A_0^t\to A_0^t/H$ is the canonical
homomorphism and $\pi^t: (A^t_0/H)^t\to A_0$ is the dual of $\pi$. 

Let $P(X)\in \Z[X]$ be the minimal polynomial of the relative
Frobenius endomorphism $\pi_0\in \End(A_0)$. 
Put $S=\Z[\pi_0]=\Z[X]/(P(X))$ and $F:=S\otimes_\Z \Q$; it is
a finite-dimensional commutative  semi-simple algebra over $\Q$. 

\begin{thm}\label{41}
There is a finite $F$-module $V$, unique up to isomorphism, such that  
$V\otimes \Q_\ell\simeq V_\ell(A_0)$ as $F_\ell$-modules for all primes
$\ell \neq p$, and $V\otimes \Q_p \simeq M(A_0)\otimes_{\Z_p} \Q_p$ as
$F_p$-modules,
where $M(A_0)$ is the \dieu module of $A_0$. 
Moreover, there is a one-to-one correspondence between 
the set of quasi-isogenies $\varphi:A\to A_0$ and the set of
$S$-lattices in $V$. In this correspondence, isomorphism classes in
the isogeny class of $A_0$ are in bijection with isomorphism classes
in the 
$S$-lattices in $V$. 
\end{thm}

\begin{proof}
  Let the abelian variety $A_0$ be isogenous to $\prod_{i=1}^r
  A_i^{n_i}$, where each abelian variety $A_i$ is simple and $A_i$ is
  not isogenous to $A_j$ for $i\neq j$. Then the endomorphism algebra
  $\End^0(A):=\End(A)\otimes \Q$ is isomorphic to 
  $\prod_{i=1}^r M_{n_i}(D_i)$, where $D_i:=\End^0(A_i)$ is a
  finite-dimensional division algebra over $\Q$. The center of the
  endomorphism algebra $\End^0(A)$ is equal to $F$. If we let $F_i$ be the
  center of $D_i$, then $F=\prod_{i=1}^r F_i$. The field $F_i$ is the
  subfield of $D_i$ generated by the Frobenius endomorphism $\pi_i$ of
  $A_i$ over $\Q$.
  Let $e_i:=[D_i:F_i]^{1/2}$ and $m_i:=n_i e_i$. Then, by a theorem of
  Tate \cite[Main Theorem, p.~134]{tate:eav}
  \[ \End^0(A_i^{n_i})\otimes \Q_\ell=M_{m_i}(F_i\otimes
  \Q_\ell) \simeq \End_{\Q(\pi_i)\otimes \Q_\ell}(V_\ell(A_i^{n_i})). \]
  This shows that the Tate space $V_\ell(A_i^{n_i})$ is a free
  $F_i\otimes \Q_\ell$ of rank $m_i$. It follows from another theorem of
  Tate \cite[Theorem, p.~525]{waterhouse:thesis} that 
\[ \End^0(A_i^{n_i})\otimes \Qp=M_{n_i}(D_i\otimes_\Q \Q_p) 
  \simeq \End_{\Q[\pi_i]\otimes \Q_p}
  (M(A_i^{n_i})\otimes \Qp). \] 
  This shows that the \dieu space $M(A_i^{n_i})\otimes \Q_p$ is a free
  $F_i\otimes \Qp$ of rank $m_i$.
  The integer $m_i$ is independent
  of $\ell$ and $p$. Put $V:=\oplus F_i^{m_i}$ as a finite $F$-module,
  and then $V$ has the desired property. 

  For any abelian variety $A$, we write $T(A):=M(A)\times
  \prod_{\ell\neq p} T_\ell(A)$. We fix an isomorphism 
\[ (*) \quad  T(A_0)\otimes  \A_f\simeq V\otimes \A_f\] 
  as $F\otimes_\Q \A_f$-modules. Then
  $T(A_0)$ is an $S\otimes \hat \Z$-lattice and there is an $S$-lattice
  $M$ in the vector space $V$ such that 
  $M\otimes_\Z \otimes \hat \Z$ is equal to
  $T(A_0)$ under the fixed rational isomorphism. 

  Let $\varphi:A\to A_0$ be a quasi-isogeny. Then the image
  $L:=\varphi_*(T(A))\subset T(A_0)\otimes \A_f\simeq V\otimes \A_f$
  is an $S\otimes \hat \Z$-lattice. Two quasi-isogenies $\varphi_1$,
  $\varphi_2$ induce the same lattice $L$ if and only if they are the
  same. Conversely, given an $S\otimes \hat \Z$-lattice $L$ in
  $V\otimes \A_f$, then by a theorem of Tate,  
  there is an abelian variety $A$ together with a quasi-isogeny
  $\varphi:A\to A_0$ such that the image $\varphi_*(T(A))$ is equal to
  $L$ under the isomorphism $(*)$. Since there is a national
  one-to-one correspondence between the set of $S$-lattices in $V$ and
  the set of $S\otimes \hat \Z$-lattices in $V\otimes \A_f$. We proved
  the second statement. The last statement follows from the basic
  fact: if a quasi-isogeny $\varphi:A_1\to A_2$ induces an
  isomorphism $T(A_1)\simeq T(A_2)$, then $\varphi:A_1\to A_2$ is an
  isomorphism. This proves the theorem. \qed
\end{proof}

It should be clear that the proof of
Theorem~\ref{41} is straightforward. However, 
it seems that this simple result has not yet been used to classify 
abelian varieties up to isomorphism (over $\Fp$).

\section{Proof of Theorem~\ref{11}}
\label{sec:05}

\subsection{Proof of Theorem~\ref{11}}
\label{sec:51}
We keep the notation as in \S~1 and 2. Since every abelian variety
over $\Fp$ isogenous to $E_0^g$ is superspecial (Lemma~\ref{22}), the
set $\scrS$ classifies isomorphism classes of abelian varieties $A$
over $\Fp$ in the isogeny class of the abelian varieties $E_0^g$. By
Theorem~\ref{41}, the set $\scrS$ is in bijection with the set of
isomorphism classes of $R$-lattices in the vector space $V=E^g$. 

Recall that a genus of $R$-lattices in $V$ is a maximal set of
$R$-lattices  
in which any two $R$-lattices are mutually isomorphic locally 
everywhere. 

\begin{lemma}\label{51}
Let $n\ge 1$ be an integer, and $K$ be an open compact subgroup of
$\GL_n(\A_{E,f})$, where $\A_{E,f}$ is the finite adele ring of the
field $E=\Q(\sqrt{-p})$. Then the determinant map $\det$ induces an
bijection of double coset spaces
\[ \det:\GL_n(E)\backslash \GL_n(\A_{E,f})/K\to E^\times \backslash
\A^\times_{E,f}/\det(K). \]  
\end{lemma}
\begin{proof}
  We may assume that $n\ge 2$. 
  Clearly that the induced map is surjective. We show the
  injectivity. Let $[a]$ be an element in the target space. Fix a section
  $s:\A_{E,f}^\times \to \GL_n(\A_{E,f})$ of the determinant map. Then
  the inverse image $T_{[a]}$ of the class $[a]$ consists of elements
  $_{\GL_n(E)} [g s(a)]_K$ for all $g\in \SL_n(\A_{E,f})$. The
  surjective map $g\mapsto _{\GL_n(E)} [g s(a)]_K$ induces a
  surjective map 
\[  \alpha: \SL_n(E)\backslash \SL_n(\A_{E,f})/K' \to T_{[a]}, \]
where $K':=s(a) K s(a)^{-1} \cap \SL_n(\A_{E,f})$. 
Since the group $\SL_n$ is simply connected and $\SL_n(E\otimes \R)$ 
is not compact, the strong
approximation holds for the algebraic group $R_{E/\Q}\SL_{n,E}$. 
Therefore, $T_{[a]}$ consists of single elements and one proves 
the lemma. \qed   
\end{proof}

Case (i): $p=2$ or $p \equiv 1 \ (\,{\rm mod}\, 4)$. In this case, the
ring $R$ is 
the maximal order in $E$. Since any $R_v$-lattice in $V_v$ for a
finite place $v$ of $\Q$ is free,
there is only one genus of $R$-lattices in $V$. The set of equivalence
classes in this genus is expressed as the double coset space
$ \GL_g(E)\backslash \GL_g(\A_{E,f})/\GL_g(\hat O_E)$. By
Lemma~\ref{51}, this double coset space is isomorphic to 
$E^\times \backslash \A^\times_{E,f}/\hat O_E^\times$ 
and hence has the cardinality $h(\sqrt{-p})$.  \\

Case (ii): $p\equiv 3 \ (\,{\rm mod}\, 4)$. 
In this case, the ring $R$ has index 2
in the maximal order $O_E$. At the place where $v\neq 2$, the ring
$R_v$ is the maximal order, and hence any two $R_v$-lattice in $V_v$
are isomorphic. At the place where $v=2$, by Corollary~\ref{32} 
there are $g+1$ isomorphism classes of $R_2$-lattices in $V_2$, namely
$R_2^r\oplus O_{E_2}^{g-r}$ for $r=0,\dots, g$. Therefore, there are
$g+1$ genera of $R$-lattices in $V$; those are represented by
$L_r:=R^r\oplus O_E^{g-r}$ for $r=0,\dots,g$. Let $K_r$ be the open
compact subgroup of $\GL_g(\A_{E,f})$ which stabilizes the $R\otimes
\hat \Z$-lattice $L_r\otimes \hat \Z$. Then by Lemma~\ref{51}, we have
\begin{equation}
  \label{eq:51}
  |\scrS|=\sum_{r=0}^g h_r,\quad h_r=\# E^\times \backslash
   \A_{E,f}^\times /\det(K_r).  
\end{equation}
It is easy to see that 
\[ \det(K_r)=
\begin{cases}
  \hat O_E^\times, & r\neq g, \\
  \hat R^\times, & r=g. 
\end{cases} \]
Therefore, $h_r=h(\sqrt{-p})$ for $r=0,\dots,g-1$. 

In the case $r=g$, we have an exact sequence of finite abelian groups:
\[ 1 \to \hat O_E^\times/(\hat O_E^\times\cap E^\times \hat R^\times)\to
\A^\times_{E,f}/E^\times \hat R^\times \to
\A^\times_{E,f}/E^\times \hat O_E^\times \to 1.\]
Since $\hat R^\times \subset \hat O_E^\times$, we have $\hat
O_E^\times\cap E^\times \hat R^\times= O_E^\times \hat R^\times$, and
hence 
\[ h_g=[\hat O_E^\times:O_E^\times \hat R^\times]\, h(\sqrt{-p}). \]
Since $\hat R_2^\times =1+2\hat O_{E_2}$, one has $\hat O_E^\times/\hat
R^\times =(O_E/2O_E)^\times$, which is $1$ or $(\F_4)^\times$ depending
on whether $2$ splits in $E$ (the case where $p\equiv 7 \ (\!\!\mod 8)$) 
or is inert in $E$ (the case where $p\equiv 3 \ (\!\!\mod 8)$). 
On the other hand,
the image of the map $O_E^\times \to (O_E/2O_E)^\times$ is $1$ except
when $p=3$. In the case where $g=3$ this map is surjective.    
It then follows that 
\[ [\hat O_E^\times:O_E^\times \hat R^\times]=
\begin{cases}
  1, & \text{if $p\equiv 7 \ (\,{\rm mod}\, 8)$ or $p=3$},\\
  3, & \text{if $p\equiv 3 \ (\,{\rm mod}\, 8)$ and $p\neq 3$}. 
\end{cases} \]
We conclude
\begin{equation}
  \label{eq:52}
  |\scrS|=
\begin{cases}
  (g+1)h(\sqrt{-p}), & \text{if $p\equiv 7 \ (\,{\rm mod}\, 8)$ 
   or $p=3$},\\
  (g+3)h(\sqrt{-p}), & \text{if $p\equiv 3 \ (\,{\rm mod}\, 8)$ and
  $p\neq 3$}.  
\end{cases}
\end{equation}
Combining Cases (i) and (ii), Theorem~\ref{11} is proved.

\subsection{}
\label{sec:52}
As a final remark, we discuss a bit about Hecke orbits in our
case. An important reference is Chai \cite{chai:ho}, where both $\ell$-adic
and prime-to-$p$ Hecke orbits in Siegel modular varieties 
are defined and explored . 

\begin{defn}
Let $k_0$ be a field, and let $S$ be a set of abelian varieties over
$k_0$. Let $A_0$ be an abelian variety in the set $S$, and $\ell$ be a
rational prime, not necessarily different from the \ch of $k_0$. For a
field extension $k$ of $k_0$, we define
{\it the $\ell$-adic Hecke orbit of $A_0$ over $k$ in $S$} as the
subset of $S$ consisting of all abelian varieties $A$ in $S$ such that
there is an $\ell$-quasi-isogeny from $A$ to $A_0$ over $k$. An
{\it $\ell$-quasi-isogeny} $\varphi:A_1\to A_2$ of two abelian varieties is
a quasi-isogeny such that there is an integer $m\in \N$ such that
$\ell ^m \varphi$ is an isogeny of $\ell$-power degree.  
\end{defn}

We retain the notation as in \S~1 and 2. We have $\scrS=\isog(E_0^g)$
and a natural map $\scrQ:=\qisog(E_0^g)\to \scrS$ which sends
any quasi-isogeny $(\varphi:A\to E_0^g)$ to $[A]$,
where the sets $\isog(E_0^g)$ and $\qisog(E_0^g)$ are defined in
\S~4, and $[A]$ denotes the isomorphism class of $A$ over $\Fp$.  
By Theorem~\ref{41}, there is a one-to-one correspondence between the
set $\scrQ$ and the set $\calL$ of $R$-lattices in $V=E^g$. Under this
correspondence, the set $\scrS$ is in bijection with the set
$\calL/\simeq$ of
isomorphism classes of $R$-lattices of $V$. Let $[L]$ denote the
isomorphism classes of an $R$-lattice $L$ in $V$. 
Suppose that $B$ is an abelian variety in $\scrS$. We choose a
quasi-isogeny $\varphi_0:B\to E_0^g$ and let $L_{\varphi_0}$ be the
$R$-lattice corresponding to $\varphi_0$. If $A$ is an abelian variety
in $\scrS$ such that there is an $\ell$-quasi-isogeny $\varphi$ from
$A$ to $B$ and let $L$ be the corresponding $R$-lattice of the
quasi-isogeny $\varphi_0\circ \varphi:A\to E_0^g$, 
then we have the relative index $[L_{\varphi_0}:L]=\ell^m$ 
for some $m\in \Z$. Recall that
$[L_{\varphi_0}:L]:=[L_{\varphi_0}:L'][L:L']^{-1}$ for any $R$-lattice
$L'$ contained in $L_{\varphi_0}\cap L$. From this, the $\ell$-adic
Hecke orbit of $B$ over $\Fp$ in $\scrS$ corresponds to the following
set
\begin{equation}
  \label{eq:53}
  \calH_\ell([L_{\varphi_0}]):=\{ [L]\in \calL/\simeq \,|\,\
  [L_{\varphi_0}:L]=\ell^m\ \text{for some $m\in \Z$}\, \}.
\end{equation}

In the case $p\equiv 3\ (\!\!\mod 4)$, there are $g+1$ genera
$\calL_0, \dots, \calL_g$ in
$\calL$ represented by the $R$-lattices $L_r=R^r\oplus O_E^{g-r}$ for
$r=0,\dots, g$. 
We further assume that $\ell\neq 2$. Then any two $R_\ell$-lattices
in $V_\ell$ are isomorphic, and 
we have 
\[ \calH_\ell([L_r])\simeq 
\Gamma_{1/\ell} \backslash \GL_g(E\otimes \Q_\ell)/K_{r,\ell}, \]    
where $\Gamma_{1/\ell}:=GL_g(E)\cap \prod_{\ell'\neq\ell} 
K_{r,\ell'}$ and $K_{r,v}$ is the $v$-component of the open compact subgroup
$K_r$. The inclusion $\calH_\ell([L_r])\subset \calL_r/\simeq $ is given by
\[  \Gamma_{1/\ell} \backslash \GL_g(E\otimes \Q_\ell)/K_{r,\ell}
\subset \GL_g(E)\backslash \GL_g(\A_{E,f})/K_r, \]
and we have $\calH_\ell([L_r])=\calL_r/\simeq$ if and only if 
\begin{equation}
  \label{eq:54}
  \GL_g(E)\backslash \GL_g(\A^{\ell}_{E,f})/K_r=1.
\end{equation}
By Lemma~\ref{51}, this is equivalent to $E^\times \backslash
\A^{\ell,\times }_{E,f}/\det(K_r)=1$, or $\Pic(O_E[1/\ell])=1$ if
$r\neq g$ and $\Pic(R[1/\ell])=1$ if $r=g$. Recall that the Picard
group $\Pic(R')$ of a commutative ring $R'$ is the group of
isomorphism classes of locally free $R'$-modules of rank one. 
Note that the condition $\Pic(R[1/\ell])=1$ implies $\Pic(O_E[1/\ell])=1$.
\begin{prop}\label{53}
  Suppose that $p\equiv 3\ (\!\!\mod 4)$ and $\ell$ is an odd
  prime. If $\Pic(R[1/\ell])\\=1$, then 
  there are $g+1$ $\ell$-adic Hecke orbits over $\Fp$ in the set 
  $\scrS$. 
\end{prop}


\begin{thebibliography}{10}

\def\jams{{\it J. Amer. Math. Soc.}} 
\def\invent{{\it Invent. Math.}} 
\def\ann{{\it Ann. Math.}} 
\def\ihes{{\it Inst. Hautes \'Etudes Sci. Publ. Math.}} 

\def\ecole{{\it Ann. Sci. \'Ecole Norm. Sup.}}
\def\ecole4{{\it Ann. Sci. \'Ecole Norm. Sup. (4)}}
\def\mathann{{\it Math. Ann.}} 
\def\duke{{\it Duke Math. J.}} 
\def\jag{{\it J. Algebraic Geom.}} 
\def\advmath{{\it Adv. Math.}}
\def\compos{{\it Compositio Math.}} 
\def\ajm{{\it Amer. J. Math.}}
\def\grenoble{{\it Ann. Inst. Fourier (Grenoble)}}
\def\crelle{{\it J. Reine Angew. Math.}}
\def\mrt{{\it Math. Res. Lett.}}
\def\imrn{{\it Int. Math. Res. Not.}}
\def\acad{{\it Proc. Nat. Acad. Sci. USA}}
\def\tams{{\it Trans. Amer. Math. Sci.}}
\def\cras{{\it C. R. Acad. Sci. Paris S\'er. I Math.}} 
\def\mathz{{\it Math. Z.}} 
\def\cmh{{\it Comment. Math. Helv.}}
\def\docmath{{\it Doc. Math. }}
\def\asian{{\it Asian J. Math.}}
\def\jussieu{{\it J. Inst. Math. Jussieu}}

\def\manmath{{\it Manuscripta Math.}} 
\def\jnt{{\it J. Number Theory}} 
\def\ijm{{\it Israel J. Math.}}
\def\ja{{\it J. Algebra}} 
\def\pams{{\it Proc. Amer. Math. Sci.}}
\def\smfmemoir{{\it Bull. Soc. Math. France, Memoire}}
\def\bsmf{{\it Bull. Soc. Math. France}}
\def\sb{{\it S\'em. Bourbaki Exp.}}
\def\jpaa{{\it J. Pure Appl. Algebra}}
\def\jems{{\it J. Eur. Math. Soc. (JEMS)}}
\def\jtokyo{{\it J. Fac. Sci. Univ. Tokyo}}
\def\cjm{{\it Canad. J. Math.}}
\def\jaums{{\it J. Australian Math. Soc.}}
\def\pspm{{\it Proc. Symp. Pure. Math.}}
\def\ast{{\it Ast\'eriques}}
\def\pamq{{\it Pure Appl. Math. Q.}}
\def\nagoya{{\it Nagoya Math. J.}}
\def\forum{{\it Forum Math. }}

\def\tp{{to appear}}

\newcommand{\princeton}[1]{Ann. Math. Studies #1, Princeton
  Univ. Press}

\newcommand{\LNM}[1]{Lecture Notes in Math., vol. #1, Springer-Verlag}

\bibitem{chai:ho} C.-L. Chai, Every ordinary symplectic isogeny
    class in positive characteristic is dense in the
    moduli. \invent~{\bf 121} (1995), 439--479.


\bibitem{demazure} M. Demazure, {\it Lectures on $p$-divisible groups}.
  \LNM{302}, 1972.

\bibitem{deuring} M. Deuring, Die Typen der Multiplikatorenringe
  elliptischer Funktionenk\"orper. 
  {\it Abh. Math. Sem. Hamburg}~{\bf 14} (1941), 197--272.

\bibitem{deuring:jdm50} M.~Deuring, Die Anzahl der Typen von
  Maximalordnungen einer definiten Quaternionenalgebra mit primer
  Grundzahl. {\it Jber. Deutsch. Math.}~{\bf 54} (1950). 24--41. 

\bibitem{eichler} M. Eichler, \"Uber die Idealklassenzahl total
  definiter Quaternionenalgebren. \mathz~{\bf 43} (1938), 102--109.

\bibitem{ekedahl:ss} T. Ekedahl, On supersingular curves and
   supersingular abelian varieties. {\it Math. Scand.}~{\bf 60} (1987),
   151--178.

\bibitem{gekeler:finite} E.-U. Gekeler, On finite Drinfeld
  modules. \ja~{\bf 141} (1991), 187--203.  

\bibitem{gekeler:mass} E.-U. Gekeler, On the arithmetic of some
  division algebras. \cmh~{\bf 67} (1992), 316--333.  

\bibitem{hashimoto-ibukiyama:classnumber} K. Hashimoto and
   T. Ibukiyama, On class numbers of positive definite binary
   quaternion hermitian forms, \jtokyo~{\bf 27} (1980), 549--601.




\bibitem{ibukiyama-katsura} T. Ibukiyama and T. Katsura, On
  the field of definition of superspecial polarized abelian varieties
  and type numbers. \compos~{\bf 91} (1994), 37--46.   

\bibitem{ibukiyama-katsura-oort} T. Ibukiyama, T. Katsura and F. Oort,
   Supersingular curves of genus two and class numbers.
   \compos~{\bf 57} (1986), 127--152.

\bibitem{katsura-oort:surface} T. Katsura and F. Oort, Families of
   supersingular abelian surfaces, \compos~{\bf 62} (1987), 107--167.

\bibitem{manin:thesis} Yu. Manin, Theory of commutative formal groups
  over fields of finite characteristic. 
  {\it Russian Math. Surveys}~{\bf 18} (1963), 1--80.

\bibitem{moret-bailly:p1} L. Moret-Bailly, Familles de courbes et de
   vari\'et\'es ab\'eliennes sur ${\mathbf P}^1$. S\'em. sur les
   pinceaux de courbes de genre au moins deux
   (ed. L. Szpiro). \ast~{\bf 86} (1981), 109--140.

\bibitem{oort:product} F. Oort, Which abelian surfaces are products of
  elliptic curves. \mathann~{\bf 214} (1975), 35--47.



\bibitem{shimura:1999} G. Shimura, Some exact formulas for quaternion
  unitary groups. \crelle~{\bf 509} (1999), 67--102.


\bibitem{tate:ht} J. Tate, Classes d'isogenie de vari\'et\'es
  ab\'eliennes sur un corps fini (d'apr\`es T. Honda). \sb~{352}
  (1968/69). \LNM{179}, 1971.

\bibitem{tate:eav} J. Tate, Endomorphisms of abelian varieties over
  finite fields. \invent~{\bf 2} (1966), 134--144.


\bibitem{waterhouse:thesis} W.~C.~Waterhouse, Abelian
  varieties over finite fields. \ecole4~{\bf 1969}, 521--560.  








\bibitem{yu-yu:ssd} C.-F.~Yu and J.~Yu, Mass formula for supersingular 
  Drinfeld modules. \cras~{\bf 338} (2004) 905--908.

\bibitem{yu:mass_hb} C.-F.~Yu, On the mass formula of supersingular
  abelian varieties with real multiplications. \jaums~{\bf 78} (2005),
  373--392.

\bibitem{yu:gmf} C.-F. Yu, An exact geometric mass
  formula. \imrn~{\bf 2008}, Article ID rnn113, 11 pages. 

\bibitem{yu:gamma_hb} C.-F. Yu, Irreducibility of Hilbert-Blumenthal
  moduli spaces with parahoric level structure. \crelle~{\bf 635}
  (2009), 187--211. 

\bibitem{yu-yu:mass_surface}
  C.-F. Yu and J.-D. Yu, Mass formula for supersingular abelian
  surfaces. \ja~{\bf 322} (2009), 3733--3743. 

\bibitem{yu:smf} C.-F. Yu, Simple mass formulas on Shimura varieties
  of PEL-type. To appear in \forum


\end{thebibliography}
\end{document}